\documentclass[reqno,11pt]{amsart}
\usepackage{amsmath, amssymb, amsthm}
\usepackage{graphicx}
\usepackage{xtab}
\usepackage{color}
\title{Level 17 Ramanujan-Sato Series}
\author{Tim Huber, Daniel Schultz, and DongXi Ye}
\address{
School of Mathematical and Statistical Sciences, University of Texas Rio Grande
Valley, Edinburg, Texas 78539, USA.}
\email{timothy.huber@utrgv.edu}
\address{
Department of Mathematics, Pennsylvania State University, University Park, State
College, Pennsylvania 16802, USA.}
\email{dps23@psu.edu}
\address{
Department of Mathematics, University of Wisconsin\\
480 Lincoln Drive, Madison, Wisconsin 53706, USA.}
\email{lawrencefrommath@gmail.com}
\newcommand{\mathsym}[1]{{}}

\newcommand{\smat}[4] {(\begin{smallmatrix} #1 & #2 \\ #3 &
    #4 \end{smallmatrix})}
\newcommand{\mat}[4] {\left(\begin{array}{cc} #1 & #2 \\ #3 &
    #4 \end{array}\right)}
\newcommand{\op}[1]  { \operatorname{ #1 }}

\newcommand{\bbZ}[0]  { \mathbb{Z}}
\newcommand{\bbQ}[0]  { \mathbb{Q}}

\numberwithin{equation}{section}
\newtheorem{theorem}{Theorem}[section]

\newtheorem{lemma}[theorem]{Lemma}

\newtheorem{corollary}[theorem]{Corollary}
\newtheorem{proposition}[theorem]{Proposition}

\begin{document}
\begin{abstract}
Two level $17$ modular functions
 $$r = q^2 \prod_{n=1}^{\infty}
 (1-q^{n})^{\left( \frac{n}{17} \right)},\qquad s = q^{2} \prod_{n=1}^{\infty} \frac{(1 -
     q^{17n})^{3}}{(1-q^{n})^{3}},
 $$
are used to construct 
a new class of
 Ramanujan-Sato series for $1/\pi$. The expansions are induced by modular identities similar to those level of $5$ and $13$ appearing in Ramanujan's Notebooks. A complete list of rational and quadratic series corresponding to singular values of the parameters is derived.
\end{abstract}

\maketitle
\allowdisplaybreaks

\section{Introduction}
Let $\tau$ be a complex number with positive imaginary part and set $q=e^{2\pi i \tau}$. Define
$$r=r(\tau) = q^2 \prod_{n=1}^{\infty} (1 - q^{n})^{\left( \frac{n}{17} \right)}, \quad s=s(\tau) = q^{2} \prod_{n=1}^{\infty} \frac{(1 -
     q^{17n})^{3}}{(1-q^{n})^{3}}.$$
In this paper, we derive level $17$ Ramanujan-Sato expansions for
$1/\pi$ of the form
\begin{align} \label{seventeen}
  q\frac{d}{dq} \log  s= \sum_{n=0}^{\infty} A_{n} \left (
  \frac{r(r^{2}s+8rs-r-s)}{8r^{3}s-3r^{2}s+r-s} \right )^{n}, \qquad
  \frac{1}{s} = r+\frac{1}{r}-2\sqrt{\frac{4}{r} - 4r - 15},
\end{align}
where $A_{n}$ is defined recursively. These relations are analogous to those at level $13$ and $5$ \cite{MR2981931,MR3324566},
\begin{align} \label{ye}
 q\frac{d}{dq} \log \mathcal{S}&= \sum_{n=0}^{\infty} \mathcal{A}(n)
                                 \left ( \frac{\mathcal{R}(1 -
                                 3\mathcal{R}-\mathcal{R}^{2})}{(1+\mathcal{R}^{2})^{2}}
                                 \right )^{n}, \qquad
                                 \ \ \ \ \frac{1}{\mathcal{S}} =
                                 \frac{1}{\mathcal{R}} - 3 -
                                 \mathcal{R} , \\   q\frac{d}{dq} \log
  S &= \sum_{n=0}^{\infty} a(n) \left ( \frac{R^{5}(1 -
      11R^{5}-R^{10})}{(1+R^{10})^{2}} \right )^{n}, \qquad
      \frac{1}{S} = \frac{1}{R^{5}} - 11 - R^{5}. \label{chan}
\end{align}
Here $a(n) $,  $\mathcal{A}(n)$
 are recursively defined sequences induced from differential
equations and
\begin{align} \label{rr_prod}
  \mathcal{R}(\tau) = q \prod_{n=1}^{\infty}
 (1-q^{n})^{\left( \frac{n}{13} \right)},  \qquad R(\tau) = q^{1/5} \prod_{n=1}^{\infty}
 (1-q^{n})^{\left( \frac{n}{5} \right)}.
\end{align}
Identity \eqref{chan} and explicit evaluations for $R(\tau)$ were used to formulate expansions for $1/\pi$ including
\begin{align} \label{five}
  \frac{1}{\pi} = \frac{1705}{81\sqrt{47}} \sum_{n=0}^{\infty} a(n)
  \left ( n+\frac{71}{682} \right ) \left ( \frac{-1}{15228} \right
  )^{n}, \qquad  a(n) = {2n \choose n} \sum_{j=0}^{n} {n \choose j}^{2} {n+j \choose j} .
\end{align} 
This expression is 
a generalization of $17$ such formulas stated by Ramanujan
\cite{MR988313,approx}. In each formula, the algebraic constants come from 
explicit evaluations for a modular function. 
The sequence $A_{k}$ are coefficients in a series solution to a
differential equation satisfied by a relevant modular
form. Generalizing such relations to higher levels requires finding
differential equations for modular
parameters and relevant identities. The series \cite{MR2981931,MR2994096,MR3258323,MR3324566}
have common construction for
primes 
$p-1 \mid 24$, where
$X_{0}(p)$ has genus zero.  More work remains to unify constructions for other levels.

The purpose of this paper is to construct level $17$ Ramanujan-Sato
series as a prototype for levels such that $X_{0}(N)$ has positive
genus. Central to the construction is the fact that $r$ and $s$ generate
the field of functions invariant under action by  an index two
subgroup of
$\Gamma_{0}(17)$. These constructions and singular value evaluations yield new Ramanujan-Sato expansions, including the rational series 
\begin{align}
  \label{eq:1}
  \frac{1}{\pi}                         &= \frac{1}{\sqrt{11}}\sum_{k=0}^{\infty}
                                                  A_k \frac{307+748k}{
                                                  (-21)^{k+2}}.
\end{align}

Following
\cite{MR2981931,MR2822224}, an expansion for $1/\pi$ is said to be
rational or quadratic if $C/\pi$ can be expressed as a series of algebraic
numbers of degree $1$ or $2$, respectively, for some algebraic number $C$. We derive a
complete list of series of rational and quadratic series from singular values of parameters in \eqref{seventeen}.

In the next section, we give an overview of results at levels $5$ and
$13$ that motivate the approach of the paper. Section \ref{i_d} includes an analogous
construction of level $17$ modular functions. This construction is used to motivate the differential
equation satisfied by 
\begin{align}
  \label{eq:2}
  z(\tau) &= \theta_q \log s, \qquad \theta_q:=q \frac{d}{d q},
\end{align}
with coefficients in the field $\Bbb C(x)$, where 
\begin{align} \label{eq:x}
x(\tau)&=\frac{r(r^{2}s + 8rs- r - s)}{8r^{3}s-3r^{2}s +r-s}.
\end{align} 
We conclude with Section \ref{singular} in which singular values are derived for
$x$ and used to construct a new class of series approximations for
$1/\pi$ of level $17$. We derive a complete list of values of
$x(\tau)$ with $[\Bbb Q(x(\tau)): \Bbb Q] \le 2$ within the radius of
convergence for $z$ as a powers series in $x$ and therefore provide
a complete list of linear and quadratic Ramanujan-Sato series
corresponding to $x(\tau)$.


\section{Level 5 and 13 Series}


The product $R(\tau)$, defined by \eqref{rr_prod}, is the
Rogers-Ramanujan continued fraction \cite{rogers}. Together, $R(\tau)$
and $S(\tau)$, defined by 
\begin{align}
  R(\tau) = \frac{q^{1/5}}{1 +
  \frac{q}{q+\frac{q^{2}}{1+\frac{q^{3}}{1+\cdots}}}}, \qquad S(\tau) = q \prod_{n=1}^{\infty} \frac{(1 -
     q^{5n})^{6}}{(1-q^{n})^{6}},
\end{align}
generate the field of functions invariant under the Hecke subgroup $\Gamma_{0}(5)$, also its index two subgroup $$\Gamma_0^2(5) = \{ \smat abcd \in \op{SL}_2(\bbZ) \mid c \equiv 0 \bmod 5 \mbox{ and } \chi(d)=1 \}\text{.}$$
This motivates 
Ramanujan's reciprocal identity \cite{ramnote}, \cite[p. 267]{berndt91}
\begin{align} \label{s}
  \frac{1}{R^{5}}-11-R^{5} = \frac{1}{S}. 
\end{align}
Equation \eqref{chan} expresses the
logarithmic derivative $Z = \theta_{q} \log S(\tau)$ in terms of a series solution to a third order linear differential equation
\cite{MR2981931}  satisfied by $S$ and the $\Gamma_{0}(5)$ invariant function $T = T(\tau)$
\begin{align} \label{deqS}
(16 T^2+44 T-1)Z_{TTT} &+ (48 T^2+66 T)Z_{TT}  + (44 T^2+34 T) Z_{T} +
  (12 T^2+6 T)Z = 0,
\\ & Z_{T} = T \frac{d}{dT} Z, \qquad T = \frac{R^{5}(1 - 11R^{5} - R^{10})}{(1 + R^{10})^{2}}.
\end{align}
The form of the equation may be anticipated from a general theorem
\cite{MR743544} (c.f. \cite{MR2031441}).
\begin{theorem} \label{ex_deq}
  Let $\Gamma$ be subgroup of $SL_{2}(\Bbb R)$ commensurable with $SL_{2}(\Bbb Z)$. If $t(\tau)$ is a nonconstant meromorphic modular function and $F(t(\tau))$ is a meromorphic modular form of weight $k$ with respect to $\Gamma$, then $F, \tau F, \ldots, \tau^{k} F$ are linearly independent solutions to a $(k+1)$st order differential linear equation with coefficients that are algebraic functions of $t$. The coefficients are polynomials when $\Gamma \setminus \mathfrak{H}$ has genus zero and $t$ generates the field of modular functions on $\Gamma$.  
\end{theorem}
Therefore, from \eqref{deqS},
\begin{align}
  \label{eq:19}
  Z = \sum_{k=0}^{\infty} a(n) T^{n},\qquad  |T| < \frac{5 \sqrt{5} -11}{8},
\end{align}
where $a(n)$ is recursively determined from \eqref{deqS}, and
expressible in closed form \cite{MR2981931} in terms of the summand
appearing in \eqref{five}.
The final ingredient needed for Ramanujan-Sato series
at level 5 are explicit evaluations for the Rogers-Ramanujan
continued fraction within the radius of convergence of the power
series. Such singular values for $R(\tau)$ were given by Ramanujan in
his first letter to Hardy \cite{MR1353909} and can be derived from modular equations satisfied by $T(\tau)$ and $T(n\tau)$. We
provide a general approach in Section \ref{singular}.

To formulate the analogous construction at level $13$, define $\mathcal{R} = \mathcal{R}(\tau)$ by \eqref{rr_prod} and
\begin{align}
  \mathcal{S}(\tau) = q \prod_{n=1}^{\infty} \frac{(1 -q^{13n})^{2}}{(1-q^{n})^{2}},\qquad  \mathcal{T}(\tau) =
  \frac{\mathcal{R}(1 - 3\mathcal{R} - \mathcal{R}^{2})}{(1 +
  \mathcal{R}^{2})^{2}}.
\end{align}
A third order linear differential equation  \cite{MR3324566} is
satisfied by the Eisenstein series $ \mathcal{Z}(\tau) = \theta_{q}
\log \mathcal{S}$ with coefficients that are polynomials in the
$\Gamma_{0}(13)$ invariant  function $\mathcal{T}(\tau)$.
For both the level 5 and 13 cases, the weight zero functions $T$ and
$\mathcal{T}$ may be uniformly presented as the quotient of a weight $4$ cusp
form and the square of a weight $2$ Eisenstein series 
\begin{align}
  \label{eq:8}
  \mathcal{T} = \frac{\mathcal{U} \mathcal{V}}{\mathcal{Z}^{2}},
  \qquad T = \frac{U V}{Z^{2}},
\end{align}
where $\mathcal{U}(\tau) = \theta_{q} \log \mathcal{R}$, $U(\tau) =\theta_{q} \log R$,
\begin{align}
  \label{eq:7}
 \mathcal{V}(\tau) = \sum_{n=1}^{\infty} \left ( \frac{n}{13} \right)
  \frac{q^{n}}{(1 - q^{n})^{2}}, \quad  V(\tau) = \sum_{n=1}^{\infty} \left ( \frac{n}{5} \right)
  \frac{q^{n}}{(1 - q^{n})^{2}}.
\end{align}

Both levels require singular values for $\mathcal{T},T$ \cite{MR3173188}. Explicit evaluations for $\mathcal{Z}$ and $Z$ follow from 
\begin{align}
  \label{eq:11}
 \mathcal{W} = \frac{\theta_{q} \log \mathcal{T}}{\mathcal{Z}} = \sqrt{1 - 12 \mathcal{T} - 16\mathcal{T}^{2}}, \qquad
  W = \frac{\theta_{q} \log T}{Z} = \sqrt{1-44T +16T^{2}}.
\end{align}

The pairs $(\mathcal{T}, \mathcal{W})$, $(T,W)$, respectively, generate the field of invariant functions for $\Gamma_{0}(13)$, $\Gamma_{0}(5)$, and $r$ and $s$  generate invariant
function fields for the congruence subgroup $\Gamma_{0}(17)$. 
\begin{proposition} \label{prop_level_gens} Let $A_{0}(\Gamma)$ denote the field of functions
  invariant under $\Gamma$
 and denote by $\chi$ the real quadratic character modulo
  $p$. Let $\Gamma_{0}^{2}(p)$ denote a index two subgroup of $\Gamma_{0}(p)$ by
  $$\Gamma_0^2(p) = \{ \smat abcd \in \op{SL}_2(\bbZ) \mid c \equiv 0 \bmod p \mbox{ and } \chi(d)=1 \}\text{.}$$
    Then 
\begin{enumerate}
 \item $A_{0}(\Gamma_{0}(5)+) = \mathbb{C}(T)$ and $A_0(\Gamma_0(5)) = \mathbb{C}(T,W)$.
\item $A_{0}(\Gamma_{0}(13)+) = \mathbb{C}(\mathcal{T})$ and
  $A_0(\Gamma_0(13)) = \mathbb{C}(\mathcal{T},\mathcal{W})$.
\item For prime $p\equiv 1 \pmod{4}$, $A_{0}(\Gamma_{0}^{2}(p)) = \mathbb{C}(R_{p}, S_{p})$, where
  \begin{align}
    \label{eq:12}
    R_{p} = q^{\ell_{p}} \prod_{n=1}^{\infty} (1 - q^{n})^{\chi(n)},
    \qquad S_{p} = q^{a_{p}} \prod_{n=1}^{\infty} \frac{(1 -
    q^{np})^{b_{p}}}{(1 - q^{n})^{b_{p}}}, \\ 
\ell_{p} = \sum_{n=1}^{\frac{p-1}{2}} \frac{n(n-p)}{2 p} \chi(n), \qquad \frac{p-1}{24} =
    \frac{a_{p}}{b_{p}}, \qquad \gcd(a_{p}, b_{p}) = 1. 
  \end{align}

\end{enumerate}
\end{proposition}
A proof of the first two parts of Proposition \ref{prop_level_gens} may be given along the
lines of the proof of Proposition \ref{prop_level17_gens}. The third
part of the Proposition is a main result of \cite{MR3544515}. 
The results of \cite{MR3544515} explain Ramanujan's level $5$ reciprocal
relation \eqref{s} and his level
$13$ reciprocal relation \cite[Equation (8.4)]{berndt91}
\begin{align}
  \frac{1}{\mathcal{R}} - 3 - \mathcal{R} = \frac{1}{\mathcal{S}}.
\end{align}
For our present work at level $17$, we apply a new identity proven in
\cite{MR3544515}
\begin{align} \label{sn}
r+\frac{1}{r}-2\sqrt{\frac{4}{r} - 4r - 15} =  \frac{1}{s}.
\end{align}
Our next task is to construct functions analogous to $T$
and $W$ in terms of $r, s$ and Eisenstein
series. 

\section{Functions invariant under $\Gamma_0(17)$ and a differential
  equation} \label{i_d}

In this section we prove an analogue to Proposition \ref{prop_level_gens}  and derive a second order linear differential equation for $z$ defined by \eqref{eq:2} with
coefficents in $\Bbb C(x)$, where $x$ is defined by \eqref{eq:x}.
In order to construct functions that are invariant under
$\Gamma_0(17)$, we introduce sums of eight Eisenstein series
considered in \cite{symmetry}. Set
\begin{align}
  \label{eq:13}
  \mathcal{E}_{1}(\tau):= \frac{1}{8}\sum_{\chi(-1) = -1} E_{\chi,k}(\tau),\qquad E_{\chi,k}(\tau) = 1 +  \frac{2}{L(1 - k, \chi)} \sum_{n=1}^{\infty} \chi(n) \frac{n^{k-1} q^{n}}{1 - q^{n}},
\end{align}
where the sum in \eqref{eq:13} is over the odd primitive Dirichlet characters modulo $17$ and $L(1 - k, \chi)$ is the analytic continuation of the associated
Dirichlet $L$-series and $\chi(-1) = (-1)^{k}$.
For $a \in (\Bbb Z / 17 \Bbb Z)^{*}$, apply the diamond operator \cite{MR2112196} to define, for  $1 \le k \le 8$, 
\begin{align} \allowdisplaybreaks
  {\langle a \rangle}  \mathcal{E}_{1}(\tau) = \frac{1}{8}\sum_{\chi(-1) = -1}
  \chi(a)E_{\chi,1}(\tau), \qquad \mathcal{E}_{k}(\tau)
= \pm {\langle 3 \rangle}^{k-1} \mathcal{E}_{1}(\tau)\text{.} \label{3}
\end{align}
The sign in Equation \eqref{3} is chosen so that the first
coefficient in the $q-$series expansion is $1$. The parameters $
\mathcal{E}_{k}(\tau)$ have the product
representations \cite[Theorems 3.1-3.5]{symmetry}
\begin{equation} \allowdisplaybreaks
\label{def_Es}
\begin{alignedat}{5}
 \mathcal{E}_{1}(\tau) &= \ \ \ \left( \begin{array}{c} q^{8},q^{ 9},q^{17},q^{17} \\ q^{2},q^{3},q^{14},q^{15}  \end{array} ; q^{17} \right)_{\infty}\text{,} \quad
&\mathcal{E}_{2}(\tau) &= q \left( \begin{array}{c} q^{3},q^{14},q^{17},q^{17} \\ q,q^{5},q^{12},q^{16}  \end{array} ; q^{17} \right)_{\infty}\text{,} \\
 \mathcal{E}_{3}(\tau) &= q^3 \left( \begin{array}{c} q,q^{16},q^{17},q^{17} \\ q^{4},q^{6},q^{11},q^{13}  \end{array} ; q^{17} \right)_{\infty}\text{,} \quad
&\mathcal{E}_{4}(\tau) &= q \left( \begin{array}{c} q^{6},q^{11},q^{17},q^{17} \\ q^{2},q^{7},q^{10},q^{15}  \end{array} ; q^{17} \right)_{\infty}\text{,} \\
 \mathcal{E}_{5}(\tau) &= q^3 \left( \begin{array}{c} q^{2},q^{15},q^{17},q^{17} \\ q^{5},q^{8},q^{ 9},q^{12}  \end{array} ; q^{17} \right)_{\infty}\text{,}\quad
&\mathcal{E}_{6}(\tau) &= q \left( \begin{array}{c} q^{5},q^{12},q^{17},q^{17} \\ q^{3},q^{4},q^{13},q^{14}  \end{array} ; q^{17} \right)_{\infty}\text{,} \\
 \mathcal{E}_{7}(\tau) &= q \left( \begin{array}{c} q^{4},q^{13},q^{17},q^{17} \\ q,q^{7},q^{10},q^{16}  \end{array} ; q^{17} \right)_{\infty}\text{,}\quad
&\mathcal{E}_{8}(\tau) &= q^2 \left( \begin{array}{c} q^{7},q^{10},q^{17},q^{17} \\ q^{6},q^{8},q^{ 9},q^{11}  \end{array} ; q^{17} \right)_{\infty}\text{,}
\end{alignedat}
\end{equation}
\begin{align*}
   \left( \begin{array}{c} a_{1} ,\ldots, a_{m} \\ b_{1}, \ldots, b_{n}   \end{array} ; z \right)_{\infty} = \prod_{n=1}^{\infty} \frac{(a_{1}; z)_{\infty} \cdots (a_{m}; z)_{\infty}}{(b_{1}; z)_{\infty} \cdots (b_{n}; z)_{\infty}}, \quad (a; z)_{\infty} = \prod_{n=0}^{\infty} (1 - az^{n}).
\end{align*}
A function $\Omega$ is now introduced as a level $17$
analogue to the level 5 cusp form $UV$.  Define
\begin{equation}
\label{def_xw}
\begin{alignedat}{3} 
\Omega(\tau) &=  \mathcal{E}_1 \mathcal{E}_2
-\mathcal{E}_2 \mathcal{E}_3
+\mathcal{E}_3 \mathcal{E}_4
-\mathcal{E}_4 \mathcal{E}_5
+\mathcal{E}_5 \mathcal{E}_6
-\mathcal{E}_6 \mathcal{E}_7
-\mathcal{E}_7 \mathcal{E}_8
-\mathcal{E}_8 \mathcal{E}_1\text{.}\\
\end{alignedat} 
\end{equation}

Proposition \ref{prop_level17_basis} demonstrates that the weight two parameters $z$ and $\Omega$, respectively, play
a role at level 17 analogous to that played by the parameters $Z$ and $UV$ at level 5. 

\begin{proposition}
\label{prop_level17_basis}
Let $z= z(\tau)$ be defined by \eqref{eq:2}. Then
\begin{enumerate}
 \item The Eisenstein space of weight two $E_2(\Gamma_0(17))$ is generated by $z$.
 \item The space of cusp forms of weight two $S_2(\Gamma_0(17))$ is generated by $\Omega$.
 \item Both $z$ and $\Omega$ change sign under $|_{W_{17},2}$, where  $f |_{W_{17},k}(\tau) =  17^{-k/2} \tau^{-k} f (-1/17\tau) \text{.}$
 \item Both $z$ and $\Omega$ have zeros at the elliptic points $\rho_{\pm}$, and in the case of $\Omega$, the zeros are simple.
\end{enumerate}
\end{proposition}

\begin{proof}
From \eqref{3} and the definition of the $\mathcal{E}_{i}$, if
arithmetic is performed modulo $8$ on the
subscripts,
\begin{align} \label{g}
  \langle 3\rangle \mathcal{E}_k = \epsilon_k \mathcal{E}_{k+1}, \qquad 
\epsilon_1, \dots, \epsilon_8 = +1,-1,+1,-1,+1,-1,-1,-1.
\end{align}
This, coupled with the transformation formula for Eisenstein series, 
\begin{align} \label{t}
   \mathcal{E}_k \left ( \frac{a \tau + b}{c \tau +d} \right ) = (c \tau + d) \cdot \langle a\rangle \mathcal{E}_k(\tau), \qquad   \begin{pmatrix}
a & b \\ c & d     
   \end{pmatrix} \in \Gamma_{0}(17)
\end{align}
implies that $\Omega$ and $z$ are modular forms of weight two with
respect to $\Gamma_{0}(17)$. From their $q$-expansions, we deduce that $\Omega$ and $z$ are
linearly independent over $\Bbb C$. Therefore, from dimension formulas
for the respective vector spaces \cite{MR2112196}, we see that these parameters
generate the vector space of weight two forms for
$\Gamma_{0}(17)$.  Thus, we obtain the first two claims of Proposition \ref{prop_level17_basis}. The third claim follows from the fact that $W_{17}$ normalizes $\Gamma_0(17)$.

As fundamental domain for $\mathbb{H}/\Gamma_0(17)$ we take
$\bigcup_{k=-8}^{8} F_k(D) \cup D$, where $D$ is the usual fundamental
domain for the full modular group and $F_k(\tau)=\frac{-1}{\tau+k}$.
The two elliptic points of order $2$ are $\rho_{\pm} = F_{\pm4}(i)$. Since $\mathbb{H}/\Gamma_0(17)$ has two elliptic points of order $2$ and two cusps, the valence formula for a weight $k$ modular form $f$ on $\Gamma_0(17)$ reads as
\begin{equation*}
\operatorname{ord}_{\infty}{f}+\operatorname{ord}_{0}{f}+\frac{\operatorname{ord}_{\rho_{+}}{f}}{2}+\frac{\operatorname{ord}_{\rho_{-}}{f}}{2} + \sum_{\tau \in (\mathbb{H} - \{\rho_{\pm}\})/\Gamma_{0}(17) } {\operatorname{ord}_{\tau}{f}} =\frac{k}{12} \cdot 18
\end{equation*}
From the $q$-expansion and the fact that $\Omega$ changes sign under $|_{W_{17},2}$, we know that the two cusps are zeros of $\Omega$, so the valence formula for $f = \Omega$ reads as
\begin{equation*}
1+1+\frac{\operatorname{ord}_{\rho_{+}}{f}}{2}+\frac{\operatorname{ord}_{\rho_{-}}{f}}{2} + \sum_{\tau \in (\mathbb{H} - \{\rho_{\pm}\})/\Gamma_{0}(17) } {\operatorname{ord}_{\tau}{f}} = 3\text{.}
\end{equation*}
Since $\Omega$ changes sign under $|_{W_{17},2}$ and the fixed point of $W_{17}$ is not a zero (as one may check numerically), the zeros must come in pairs. Accordingly, the two other zeros must be the two elliptic points, and these are simple zeros. A similar argument gives the result for $z$.
\end{proof}

The cusp form and Eisenstein series from Proposition
\ref{prop_level17_basis} can now be used in the construction of a
$\Gamma_{0}(17)$ invariant function of the same form as $T, \mathcal{T}$ given by \eqref{eq:8}. Although the
representation for $x(\tau)$ given here appears to differ from that given in the introduction, we
ultimately demonstrate agreement of the two representations in
Proposition \ref{h}.  The parameters $x(\tau)$ and $w(\tau)$, defined in
Proposition \ref{prop_level17_gens}, play roles analogous to
corresponding parameters $T$ and $W$ in Proposition \ref{prop_level_gens}.

\begin{proposition}
\label{prop_level17_gens}
 If the Fricke involution is denoted $W_{17} = W_{17, 0}$ and $x$ and
 $w$ are defined by 
 \begin{align}
   \label{eq:21}
   x(\tau) = \frac{\Omega}{z}, \qquad  w(\tau) = \frac{2}{z} \theta_{q} \log x,
 \end{align}
\begin{enumerate}
 \item $x$ is invariant under $\Gamma_0(17)$ as well as $W_{17}$; and
 \item $x$ has two simple zeros on $\mathbb{H}/\Gamma_0(17)$ at the two cusps.
 \item The field of functions invariant under $\Gamma_0(17)$ and $W_{17}$ is $A_0(\langle \Gamma_0(17)+ \rangle) = \mathbb{C}(x)$.
 \item The field of functions invariant under $\Gamma_0(17)$ is given by $A_0(\Gamma_0(17)) = \mathbb{C}(x,w)$.
 \item The relation $w^2=-127 x^4-48 x^3-66 x^2-16 x+1$ holds.
\end{enumerate}
\end{proposition}

\begin{proof}
The first two assertions follow directly from Proposition \ref{prop_level17_basis}. The third assertion is then a direct consequence of the first two. For the fourth assertion, the functions $x(\tau)$ and $w(\tau)$ are invariant under $\Gamma_0(17)$, so it suffices to show that they generate the whole field. Since $x$ has order $2$, we have $\left[A_0(\Gamma_0(17)): \mathbb{C}(x) \right] =2$. Since $w \not \in \mathbb{C}(x)$ because it changes sign under $W_{17}$, we must have $\left[A_0(\Gamma_0(17)): \mathbb{C}(x,w) \right] =1$, that is, the second assertion holds. For the final assertion, the function $w^2$ is fixed under $W_{17}$ and has the same set of poles as $x$, hence it is a polynomial in $x$. We bound the degree of this polynomial by $4$ and find its coefficients by comparing $q-$expansions.
\end{proof}

The parameter $x$ 
is expressible as the rational function of $r$
and $s$ appearing in the Introduction and in terms of the
McKay-Thompson series $17A$ \cite[Table 4A]{MR554399}.
\begin{proposition} \label{h}
Define $\eta(\tau) =
q^{1/24}(q;q)_{\infty}$, and let $x$ be defined
as in Proposition \ref{prop_level17_gens}. Then 
\begin{align}
x&=\frac{r(r^{2}s + 8rs- r - s)}{8r^{3}s-3r^{2}s +r-s}, \label{XX} \\ 
\frac{1-x}{2 x} &=\frac{1}{4\eta(\tau)^2 \eta(17\tau)^2}\left(\sum_{m,n=-\infty}^{\infty}{(e^{\pi i m}-e^{\pi i n})q^{\frac{1}{4}n^2+\frac{17}{4}m^2}}\right)^2.  \label{theta}
\end{align}
\end{proposition}

\begin{proof}
From the product representation for $r$ and those for the
Eisenstein sums $\mathcal{E}_{i},$ from \eqref{def_Es}
\begin{align} \label{tf}
  r = \frac{\mathcal{E}_{1} \mathcal{E}_{3}
  \mathcal{E}_{5}\mathcal{E}_{7}}{\mathcal{E}_{2} \mathcal{E}_{4}
  \mathcal{E}_{6} \mathcal{E}_{8}}.
\end{align}
Therefore, $r$ is the quotient of weight four modular forms for $\Gamma_{1}(17)$, and $x = \Omega/z$ is the quotient of weight two modular forms for $\Gamma_{1}(17)$.
Hence, the quadratic relation between $x$ and $r$, 
\begin{align}
  \label{eq:16}
  \frac{4}{r} - 4r - 15 = \frac{(xr-1)^{2}(4r -1)^{2}}{(x+r)^{2}}
\end{align}
may be transcribed as a relation between modular forms of weight 20 for
$\Gamma_{1}(17)$ and proved from the Sturm bound by verifying the $q$-expansion to
order $481=1+20
\cdot 288/12 $. Then
\begin{align}
  \label{eq:17}
 \sqrt{ \frac{4}{r} - 4r - 15} = \frac{(xr-1)(4r -1)}{(x+r)},
\end{align}
where the branch of the square root is determined using the definition
of $x$ and $r$. Therefore, 
\begin{align}
  \label{eq:18}
  x = \frac{4r -1+ r \beta(r)}{4r^{2} - r - \beta(r)},\qquad \beta(r)
  =  \sqrt{ \frac{4}{r} - 4r - 15} .
\end{align}
The first equation of \eqref{eq:18} is seen to be equivalent to
\eqref{XX} by applying \eqref{sn}.
Equation \eqref{theta} may be derived from respective $q$-expansions since each side is a Hauptmodul for $\Gamma_{0}(17)+$. 
\end{proof}
It follows from the first part of Proposition \ref{prop_level17_gens}
and Theorem \ref{ex_deq} that $z$ satisfies a third order linear homogeneous differential
equation with coefficients in $\mathbb{C}(x)$. In order to formulate the
differential equation, we state the following preliminary nonlinear differential
equation in terms of the differential operator $\theta_q := q
\frac{q}{d q}$. This is written even more succinctly as $f_q:=\theta_q f$. 
\begin{lemma}
\label{lemma_Zq_DIFEQ}
 \begin{equation*}
  \frac{2 z z_{qq}-3 z_q^2}{3 z^4} =  \frac{x \left(127 x^5-222 x^4+126 x^3+4 x^2+27 x+2\right)}{4 (x-1)^2}
 \end{equation*}
\end{lemma}
\begin{proof}
Let $f(\tau)$ denote the function on the left hand side of the proposed equality. If $z$ satsifies the functional equation
\begin{equation*}
 z \left( \frac{a \tau+b}{c \tau+d} \right) = \epsilon \frac{(c \tau+d)^2}{a d-b c} z(\tau)
\end{equation*}
 one can compute that 
\begin{equation*}
\frac{2 z z_{qq}-3 z_q^2}{3 z^4}  \left( \frac{a \tau+b}{c \tau+d} \right) = \frac{1}{\epsilon^{2}}  \frac{2 z z_{qq}-3 z_q^2}{3 z^4} ( \tau ) \text{.}
\end{equation*}
By Proposition \ref{prop_level17_basis}, we have $\epsilon=1$ for elements of $\Gamma_0(17)$ and $\epsilon=-1$ for $W_{17}$. Thus we see that $f(\tau)$ is invariant under $\Gamma_0(17)$ and $W_{17}$ in weight $0$. According to Theorem \ref{expl}, we see that $x$ does not have a pole at the two elliptic points, i.e. $x(\rho_{\pm})=1$. This means that the two zeros of $z$ at these elliptic points are both simple. Hence, $z$ has two other simple zeros $p_1$ and $p_2=W_{17}(p_1)$, which are also the poles of $x$, modulo $\Gamma_0(17)$, as observed in the proof of Proposition \ref{prop_level17_basis}. Since all of the poles are $z$ are simple, we can take the expansion
\begin{equation*}
 z(\tau) = c (\tau-\tau_{0})+ \dots
\end{equation*}
at the zeros $\tau_{0}=\rho_{+},\rho_{-},p_1,p_2$, where $c$ is non-zero. Each of these zeros contributes a quadruple pole to $f(\tau)$ since
\begin{equation*}
 \frac{2 z z_{qq}-3 z_q^2}{3 z^4}  (\tau) = \frac{3}{(2\pi c)^2 (\tau-\tau_{0})^4} + \cdots
\end{equation*}
 In the fundamental domain of $\mathbb{H}/\Gamma_0(17)$, the translate $F_{4}(D)$ is adjacent to itself. Thus $x(\tau)$ must identify the two halves of the corresponding side of $F_{4}(D)$ (the side that contains $F_4(i)$). Likewise for $F_{-4}$. Therefore, at the elliptic point $\rho_{\pm}$, the function $x(\tau)$ is locally a holomorphic function of $((\tau-\rho_{\pm})/(\tau-\bar{\rho_{\pm}}))^2$ so that
\begin{equation*}
 x(\tau) = 1 + c_{\pm} (\tau-\rho_{\pm})^2 + \cdots
\end{equation*}
We see now that $(x-1)^2 f$ has poles only at $p_{1}$ and $p_2$, each of order six. It is therefore a polynomial of degree six in $x$, and we can compute that
\begin{equation*}
 4 (x-1)^2 f - x (127 x^5-222 x^4+126 x^3+4 x^2+27 x+2) = O(q^7)\text{.}
\end{equation*}
The left hand side has poles of order $6$ at $p_{1}, p_{2}$ and zeros at least order $7$ at $0$ and $\infty$. This contradicts the valence formula unless the left hand side is constant.
\end{proof}

We now give the third order linear differential equation for $z$ with rational coefficients in $x$. The concise formulation of the differential equation in \eqref{thm_ZxDIFEQ_equ1} is motivated by the general form of such differential equations from \cite{MR743544,MR2031441}.
\begin{theorem}
 \label{thm_ZxDIFEQ}
With respect to the function $x$, the form $f=z$ satisfies the differential equation. 
\begin{align*}
 0 &= 3 x (254 x^6-714 x^5+681 x^4-250 x^3-6 x^2-28 x-1) f\\
& + x(x-1) (1397 x^5-2482 x^4+1094 x^3-28 x^2+197 x+14) f_{x}\\
& + 6 x (x-1)^3 (127 x^3+36 x^2+33 x+4) f_{xx}\\
& + (x-1)^3 (127 x^4+48 x^3+66 x^2+16 x-1) f_{xxx}\text{.}
\end{align*}
\end{theorem}
\begin{proof}
The differential equation satisfied by $f = z$ is given as
\begin{equation} 
\label{thm_ZxDIFEQ_equ1}
 \operatorname{det} \left(
\begin{array}{cccc}
 f & f_x & f_{xx} & f_{xxx} \\
 (z) & (z)_x & (z)_{xx} & (z)_{xxx} \\
 (z\log q) & (z\log q)_x & (z\log q)_{xx} & (z\log q)_{xxx} \\
 (z\log^2 q) & (z\log^2 q)_x & (z\log^2 q)_{xx} & (z\log^2 q)_{xxx}
\end{array}
\right) =0\text{.}
\end{equation}
When expanding this determinant, we make the following substitutions:
\begin{enumerate}
 \item For the diffential with respect to $x$, use the definition \eqref{eq:21} in the form
\begin{equation*}
 \theta_x = x \frac{\partial}{\partial x} = \frac{2}{w z} \theta_q\text{.}
\end{equation*}
\item When the first derivative $x_q$ appears, use the definition \eqref{eq:21} in the form
\begin{equation*}
 x_q = \frac{1}{2}x w z\text{.}
\end{equation*}
\item When the first derivative $w_q$ appears, use the relation between $w$ and $x$ to obtain
\begin{equation*}
 w_q = -x (127 x^3+36 x^2+33 x+4) z\text{.}
\end{equation*}
\item When the second derivative $z_{qq}$ appears, use Lemma \ref{lemma_Zq_DIFEQ} in the form
\begin{equation*}
 z_{qq}=\frac{3 z_q^2}{2 z}+\frac{3 x(127 x^5-222 x^4+126 x^3+4 x^2+27 x+2)}{8(x-1)^2} z^3
\end{equation*}
\end{enumerate}
When these substitutions are made in \eqref{thm_ZxDIFEQ_equ1}, the claimed differential equation results after clearing denominators by multiplying by $(1 - x)^3 w^5/16$ and using Proposition \ref{prop_level17_gens} (5).
\end{proof}

The linear differential equation in Theorem \ref{thm_ZxDIFEQ}
induces a series expansion for $z$ in terms of $x$ with coefficients $A_{n}$. 

\begin{corollary}
 \label{cor_Zx_coeff}
 \begin{equation}
\begin{alignedat}{5}
 z &= \sum_{n=0}^{\infty} A_n x^n& \quad \quad |x| &< 0.05122\dots,
\end{alignedat}
\end{equation}
where $A_0=2$, $A_{-1,\dots,-6}=0$ and 
\begin{gather*}
0=(n+1)^3 A_{n+1}
+(-19 n^3-24 n^2-14n-3) A_n\\
-3 \left(5 n^3+27 n^2-8 n+4\right) A_{n-1}
+\left(101n^3-300 n^2+213 n-52\right) A_{n-2}\\
-3 \left(55 n^3-267 n^2+491 n-305\right) A_{n-3}
+ 3 (n-3) \left(101 n^2-297 n+253\right) A_{n-4}\\
-9 (n-4) (n-3) (37 n-66) A_{n-5}
+127 (n-5) (n-4) (n-3) A_{n-6}.
\end{gather*}
The radius of convergence is the positive root of $127 x^4+48 x^3+66 x^2+16 x-1$.
\end{corollary}

To make use of the series appearing in Corollary \ref{cor_Zx_coeff},
we require
explicit evaluations for the $x(\tau)$  within the domain of
validity. In the next section, we prove that the number of singular
values is finite and compile a complete list of quadratic evaluations and
expansions. 

\section{Singular Values and Series for $1/\pi$} \label{singular}

In this section, singular values for $x(\tau)$ are derived and used to formulate
Ramanujan-Sato expansions. The work culminates in a proof that there are  precisely $11$
singular values for $x(\tau)$ of degree at most two over $\Bbb Q$ within the radius of convergence of
Corollary \ref{cor_Zx_coeff}. The series given by \eqref{eq:1} is
the only such expansion with a rational singular value
for $x(\tau)$.
The main challenge in proving the expansions lies in rigorously determining exact evaluations for $x(\tau)$ for given $\tau$ and deriving constants appearing in the Ramanujan-Sato series. To do this, we formulate modular equations for $x(\tau)$ and provide an explicit
relation between the modular equations and constants appearing in the
series.

 We demonstrate in the proof of Theorem \ref{last} that the following table is a complete list of
 singular 
values for $x(\tau)$ in a fundamental domain for $\Gamma_{0}(17)$ with
  $[\bbQ(x(\tau)):\bbQ] \le 2$ within the radius of convergence of Corollary \ref{cor_Zx_coeff}. Each value $\tau$
  is listed by the coefficients $(a,b,c)$ of its minimal
  polynomial, and the values are ordered by discriminant. 

\begin{table}
\begin{equation}
\begin{array}{lll}
b^2-4ac & \tau(a,b,c) & x(\tau)  \\ 
 -1411 & (17,17,25) & (-1025-252 \sqrt{17})^{-1} \\
 -1003 & (17,17,19) & (-345-84 \sqrt{17})^{-1} \\
 -595 & (17,-17,13) & (-90-21 \sqrt{17})^{-1} \\
 -427 & (17,-27,17) & (30+33 i \sqrt{7})^{-1} \\
 -427 & (17,-41,31) & (30-33 i \sqrt{7})^{-1} \\
 -408 & (17,-34,23) & (55+24 \sqrt{2})^{-1} \\
 -408 & (34,-68,37) & (55-24 \sqrt{2})^{-1} \\
 -340 & (17,-34,22) & (29+4 \sqrt{85})^{-1} \\
 -323 & (17,-17,9) & (-22-7 \sqrt{17})^{-1} \\
 -187 & (17,-17,7) & -1/21 \\
 -136 & (17,-34,19) & (12+3\sqrt{17})^{-1} 
\end{array}
\end{equation}

\caption{Complete list of singular values of $x(\tau)$ of degree at most $2$ within the
  radius of convergence of Corollary 3.6, ordered by discriminant.}
\label{tab_x_values_17}
\end{table}

By Proposition \ref{h}, finding singular values for $x(\tau)$ is
equivalent to finding singular values for the normalized Thompson
series $17A$. The fundamental results needed for such evaluations are
presented in \cite{MR1400423}. Our work below is
a detailed rendition of the general
presentation in \cite{MR1400423} tailored to the modular function
$x(\tau)$. We begin with the set of matrices
\begin{equation*}
\Delta_n^*(17)=\{\smat \alpha \beta \gamma \delta \in \bbZ^{2\times2}| \gcd(\alpha,\beta,\gamma,\delta)=1 \text{ and } \alpha \delta-\beta \gamma=n \text{ and } \gamma \equiv 0 \bmod 17\}\text{.}
\end{equation*}

\begin{lemma}
\label{lem_Delta}
If $\gcd(n,17)=1$, then $\Delta_n^*(17)$ has the coset decomposition
\begin{equation*}
 \Delta_n^*(17) = \bigcup_{\underset{\underset{\gcd(\alpha,\beta,\delta)=1}{0\le\beta<\delta}}{\alpha \delta=n}} \Gamma_0(17) \mat \alpha \beta 0 \delta \text{,}
\end{equation*}
and the double coset representation
\begin{equation*}
\Delta_n^*(17) = \Gamma_0(17) \mat 1 0 0 n \Gamma_0(17)\text{.}
\end{equation*}
\end{lemma}
\begin{proof}
Any $\smat \alpha \beta \gamma \delta \in \Delta_n^*$ can be converted to an upper triangular matrix by multiplying on the left by a matrix of the form
\begin{equation*}
\mat {*}{*}{\frac{\gamma}{\gcd(\alpha,\gamma)}}{\frac{-\alpha}{\gcd(\alpha,\gamma)}} \in \Gamma_0(17)\text{.}
\end{equation*}
It is then easy to see that the claimed representatives are distinct modulo $\Gamma_0(17)$. This proves the decomposition formula. Next, by performing elementary row and column operations on the matrix $m \in \Delta_n^*(17)$, we find matrices $\gamma_1, \gamma_2 \in \Gamma(1)$ such that $m = \gamma_1 \smat 100n \gamma_2$. Since
\begin{equation*}
\mat 1 0 0 n  = \mat a b {nc} {d} \mat 1 0 0 n \mat a {-bn} {-c} {d}\text{,}
\end{equation*}
we may find appropriate $a$, $b$, $c$, and $d$ such that $\gamma_1'=\gamma_1 {\smat a b {nc} {d}} \in \Gamma_0(17)$. This results in an equality of the form $m = \gamma_1' \smat 100n \gamma_2'$ where $\gamma_1' \in \Gamma_0(17)$, which forces $\gamma_2'\in \Gamma_0(17)$ as well. This establishes the double coset representation.
\end{proof}

We now establish modular equations central to our explicit evaluations
for $x(\tau)$.

\begin{proposition}
\label{prop_modequ}
 For any integer $n \ge2 $ with $\gcd(n,17)=1$, there is a polynomial $\Psi_n(X,Y)$ of degree $\psi(n) = n \prod_{\substack{q|n \\ q \text{ prime}}}(1+\tfrac 1 q)$ in $X$ and $Y$ such that:
\begin{enumerate}
 \item $\Psi_n(X,Y)$ is irreducible and has degree $\psi(n)$ in $X$ and $Y$.
 \item $\Psi_{n}(X,Y)$ is symmetric in $X$ and $Y$.
 \item The roots of $\Psi_n(x(\tau),Y) = 0$ are precisely the numbers $Y=x((\alpha \tau+\beta)/\delta)$ for integers $\alpha$, $\beta$ and $\delta$ such that $\alpha \delta=n$, $0\le \beta < \delta$, and $\gcd(\alpha,\beta,\delta)=1$.
\end{enumerate}
\end{proposition}
\begin{proof}
The polynomial $\Psi_n$ satisfies
\begin{equation}
\label{prop_modequ_equ1}
(XY)^{-\psi(n)} \Psi_{n}(X,Y)=  \prod_{ \underset{\underset{(\alpha,\beta,\delta)=1}{0\le\beta<\delta}}{\alpha \delta=n}} \left( Y^{-1} - x\left( \frac{\alpha \tau+\beta}{\delta} \right)^{-1} \right)\text{,}
\end{equation}
where the coefficients of $Y^{-k}$ on the right hand side should
expressed as polynomials in $1/X$ for $X=x(\tau)$ as demonstrated in
the proof of Corollary \ref{mod}. This relies on the fact that $\Gamma_0(17)$ and $W_{17}$ permute the set of functions $x\left( (\alpha \tau+\beta)/\delta \right)$ where $\alpha \delta=n$, $0\le\beta<\delta$, and $\gcd(\alpha,\beta,\delta)=1$. The double coset representation in Lemma \ref{lem_Delta} shows that every orbit contains $x(\tau/n)$, and hence the action of $\Gamma_0(17)$ on the roots must be transitive. Since
\begin{equation*}
{\mat 0 {-1} {17} 0}^{-1} \mat \alpha \beta 0 \delta \mat 0 {-1} {17} 0 = \mat {\delta}{0}{-17 \beta}{\alpha} \in \Delta_n^*(17)\text{,}
\end{equation*}
it is clear that $W_{17}$ permutes these functions as well by the decomposition in Lemma \ref{lem_Delta}. The coefficient of $X^{\psi(n)}Y^{\psi(n)}$ in $\Psi_{n}(X,Y)$ is the constant term of the product on the right hand side of \eqref{prop_modequ_equ1}, which is clearly non-zero because the function $x(\tau)$ does not have poles at the cusps of $\mathbb{H}/\Gamma_0(17)$. Therefore, $\Psi_n(X,X)$ has the claimed degree $2\psi(n)$. The symmetry can be proven by noting that $\tau \to -1/(17n \tau)$ interchanges $x(\tau)$ and $x(n \tau)$.
\end{proof}

In Corollary \ref{mod}, modular equations
$\Psi_{n}(X, Y) =0$ are derived for $X = x(\tau)$ and $y = x((\alpha \tau + \beta)/\delta)$ satisfying the conditions of Proposition
\ref{prop_modequ}. The proof indicates how modular equations for
larger $n$ may be derived and involves techniques analogous to those used to deduce
classical modular equations of level $n$ satisfied by the $j$
invariant \cite{MR890960,MR0314766}. 
\begin{corollary} \label{mod}
 We have
\begin{align*}
 \Psi_{2}(X,Y)&=-9 X^3 Y^3-12 X^3 Y^2+X^3 Y+2 X^3-12 X^2 Y^3\\
              &+8 X^2 Y^2+10 X^2 Y+X Y^3+10 X Y^2-X Y+2 Y^3\text{,} \\
\Psi_{3}(X,Y) &=435X^4Y^4+231X^4Y^3+231X^3Y^4+45X^4Y^2-385X^3Y^3+45X^2Y^4\\
&-39X^4Y-63X^3Y^2-63X^2Y^3-39XY^4+4X^4+9X^3Y+123X^2Y^2+9XY^3\\
&+4Y^4+15X^2Y+15XY^2-XY. 
\end{align*}
\end{corollary}

\begin{proof}
The level $n = 2$ result is representative of $n=3$ and other cases. For $n=2$,
we have $$\{(\alpha,\beta,\delta) \mid
\gcd(\alpha,\beta,\delta) = 1, 0\le\beta<\delta, \alpha \delta=n\} =
\{(1,0,2), (2,0,1),(1,1,2)\}.$$ Let $x(\tau)$ be defined as in
Proposition \ref{prop_level17_gens}. Then $$x_{1} = x( \tau/2), \quad 
x_{2} = x(2 \tau), \quad x_{3} = x\left ( \frac{\tau +1}{2} \right).$$
By \eqref{prop_modequ_equ1}, 
\begin{align*}
\Psi_{2}(1/X, 1/Y)&= Y^{-3} - \left (\frac{x_{1} x_{2} + x_{1} x_{3} + x_{2} x_{3}}{x_{1} x_{2}
  x_{3}} \right ) Y^{-2} + \left ( \frac{ x_{1} + x_{2} +
  x_{3}}{x_{1} x_{2} x_{3}} \right )Y^{-1} - \frac{1}{x_{1} x_{2} x_{3}}.
\end{align*}
By Theorem \ref{prop_level17_gens}, we know that $1/x(\tau)$
is analytic on $X_{0}(17)$ except for simple poles at the
cusps. Therefore, the only
poles in $\Bbb H/\Gamma_{0}(17)$  of the coefficients of $Y^{-k}$ are at points
equivalent to the cusps
$0$ and $\infty$. We can explicitly compute the $q$-expansion for each
of the coefficients and deduce that each has a pole of order at most $3$ at $q
= 0$. Since each coefficient is invariant under $\Gamma_{0}(17)$ and
$W_{17}$, the coefficients may be expressed as
polynomials of degree at most $3$ in $1/x(\tau)$. Explicitly, 
\begin{align*}
 -\frac{x_{1}x_{2} + x_{1}x_{3} + x_{2} x_{3}}{x_{1} x_{2} x_{3}} &=
  -\frac{2}{q^2}-15 +O(q), \\
  \frac{x_{1} + x_{2} + x_{3}}{x_{1}x_{2}x_{3}} &=
                                                  \frac{20}{q^2}+\frac{108}{q}+419+O(q),
  \\ 
 -  \frac{1}{x_{1} x_{2} x_{3}} &=
  \frac{8}{q^3}+\frac{62}{q^2}+\frac{316}{q}+1307 + O(q).
 \end{align*}
Therefore, with $x = x(\tau)$, we may determine polynomials in $1/x$
such that 
\begin{align*}
c_{1}(\tau) &=  -  \frac{1}{x_{1} x_{2} x_{3}} - \left (
  -\frac{9}{2}-6x^{-1}+\frac{1}{2}x^{-2}+x^{-3}\right ) = O(q), \\ 
c_{2}(\tau) &= \frac{x_{1} + x_{2} + x_{3}}{x_{1}x_{2}x_{3}} - \left (
   -6 + 4 x^{-1} + 5 x^{-2}\right ) = O(q),\\ 
c_{3}(\tau) &= -\frac{x_{1}x_{2} + x_{1}x_{3} + x_{2} x_{3}}{x_{1}
              x_{2} x_{3}} - \left ( \frac{1}{2} + 5x^{-1} -
              \frac{1}{2} x^{-2}\right ) = O(q).
\end{align*}
Since the functions $c_{j}(\tau)$, $j=1,2,3$, are analytic on the upper half plane and at the cusps of
$X_{0}(17)$, each $c_{j}$ is constant, and $c_{j}(i\infty) = 0$. 
\end{proof}

In Theorem \ref{expl}, we prove each evaluation
from Table \ref{tab_x_values_17}. The evaluations are proven by
showing that $x(\tau)$ satisfies a modular equation of degree $n$ for
some $n$.

\begin{theorem} \label{expl}  
%
%
Each evaluation for
  $x(\tau(a,b,c))$ from Table
  \ref{tab_x_values_17} holds, and $x(\tau(17,-8,1)) =1$. 

\end{theorem}
\begin{proof}
First note that $$\tau(17,-34,19) = \frac{1}{17} \left(17+i
  \sqrt{34}\right).$$ Therefore, $$x( \tau ) = x\Bigl ( i \sqrt{\frac{2}{17}}\Bigr )$$

Observe that with $$ \tau = i \sqrt{\frac{2}{17}}, \qquad - \frac{1}{17 \tau} = \frac{\tau}{2}.$$
Since $x(\tau)$ is invariant under the Fricke involution $W_{17}$, we have
$x( \tau)=x(\tau/2)$. That is, $X=Y$ in the degree 2 modular equation above. Setting $Y=X$ and simplifying the equation, we get
\begin{align}
  \label{eq:20}
X^2(X-1)(X+1)(9X^2+24X-1)=0.  
\end{align}
Now that we have proven $x(\tau)$ satisfies \eqref{eq:20}, we may
numerically deduce $x(\tau)$ is a root of $9X^2+24X-1$, and 
$$
x\Bigl ( i \sqrt{\frac{2}{17}}\Bigr )=-\frac{4}{3}+\frac{1}{3}\sqrt{17}.
$$
 We may similarly prove $x(\tau(17,-8,1)) =1$ and the remaining evaluations in Table
  \ref{tab_x_values_17} from modular equations of degree $n$ if we can
  determine, for each given value of $\tau = \tau(a,b,c)$,
  an upper upper triangular matrix $(\alpha,\beta ; 0,\delta)$ such
  that $x(\tau)  = x((\alpha,\beta;0,\delta)
  \tau)$ and
  $\alpha \delta = n$, $0\le \beta <\delta$, $\gcd(\alpha, \beta,
  \delta) = 1$, with $\gcd(n,17) =1$. For each $\tau$,  Table \ref{two} provides a $\gamma \in
  \Delta_n^*(17)$ such that $\gamma \tau = \tau$ or $W_{17}\tau$ and a
  $\Gamma_{0}(17)$ equivalent upper triangular matrix $(\alpha,\beta;0,\delta)$.
\begin{table}
\begin{equation}
\begin{array}{lll}
 \tau(a,b,c) & \text{Element of }\Delta_{n}^{*}(17)
 & (\alpha,\beta;0,\gamma) \\ 
  (17,17,25) &   (-1, -2; 17, -25) & (1, 2; 0, 59)  \\
  (17,17,19) &   (1, 0; 17, 19) & (1, 0; 0, 19)   \\ 
  (17,-17,13) & (-1, 0; 17, -13) & (1, 0; 0, 13)  \\
  (17,-27,17) & (13, -17; 17, -14)^{\dagger} & (1, 81; 0, 107) \\
  (17,-41,31) & (20, -31; 17, -21)^{\dagger} & (1, 68; 0, 107) \\
  (17,-34,23) & (-1, 0; 34, -23) & (1, 0; 0, 23)  \\
  (34,-68,37) & (-2, 1; 51, -37) & (1, 11; 0, 23)  \\
  (17,-34,22) & (-1, 1; 17, -22) & (1, 4; 0, 5)  \\
  (17,-17,9) &   (-2, 1; 17, -18) & (1, 9; 0, 19)  \\
  (17,-17,7) &   (-1, 0; 17, -7) & (1, 0; 0, 7)  \\
  (17,-34,19) & (-1, 1; 17, -19) & (1, 1; 0, 2)  \\
  (17,-8,1) & (3,-1; 17, -5) & (1,1; 0, 2)
\end{array}
\end{equation}
\caption{Elements of $\Delta_n^*(17)$
mapping $\tau$ to its image under $W_{17}$ or fixing$^{\dagger}$ $\tau$
 and
a corresponding $\Gamma_{0}(17)$ equivalent upper triangular matrix $(\alpha,\beta;0,\gamma)$.}
\label{two}
\end{table}
\end{proof}

For values $x(\tau)$ in the domain of validity
for series from Corollary \ref{cor_Zx_coeff},
we may
construct Ramanujan-Sato expansions via
Theorem \ref{thm_pi_main}, a specialization of \cite[Theorem
2.1]{MR2073912}. 


\begin{theorem}[Series for $1/\pi$]
\label{thm_pi_main}
Suppose there is a matrix $(a,b;c,d) \in \langle \Gamma_0(17),W_{17} \rangle$ such that
\begin{equation*}
\frac{a \tau+b}{c \tau+d}= \frac{\alpha \tau+\beta}{\delta}
\end{equation*}
for $\alpha \delta=n$ and $0\le\beta<\delta$. Set $X=x(\tau)$, which
is determined from $\Psi_n(X,X)=0$, and further set $$W=w(\tau),\quad
\Psi_X=\frac{\partial \Psi_n}{\partial X}(X,X),\quad
\Psi_Y=\frac{\partial \Psi_n}{\partial Y}(X,X),$$ and let
$\epsilon\in\mathbb{Q}$ and $\eta=\pm1$ satsify for all $\tau$
\begin{equation*}
z\left( \frac{a \tau+b}{c \tau+d} \right) = \epsilon (c \tau+d)^2 z(\tau) \text{,} \quad w\left( \frac{a \tau+b}{c \tau+d} \right) = \eta w(\tau).
\end{equation*}
If $A_k$ is the sequence defined in Corollary \ref{cor_Zx_coeff} and
\begin{align*}
B&=-\frac{i W \left(\delta ^2 \Psi _X (a d-b c)+\alpha ^2
   \eta  \epsilon  \Psi _Y (c \tau +d)^4\right)}{2 \alpha ^2 c
   \eta  \epsilon  \Psi _Y (c \tau +d)^3}\text{,}\\
C&=\frac{i \delta ^2 (b c-a d) W}{2 \alpha ^2 c \eta \epsilon  \Psi
   _Y^3 (c \tau+d)^3} (\Psi _X \Psi _Y \left(\Psi _X+\Psi _Y\right)
   (1+\theta_X \log W) \\ & \qquad + (\Psi _X^2 \Psi _{\text{YY}}-2 \Psi_X \Psi _{\text{XY}} \Psi _Y+\Psi_{\text{XX}} \Psi _Y^2) X) \text{,}
\end{align*}
then
\begin{equation*}
\frac{1}{\pi} = \sum_{k=0}^{\infty}  A_k (B k+C) X^k\text{.}
\end{equation*}
\end{theorem}
\begin{proof}
Differentiate the relation
\begin{equation*}
z\left( \frac{a \tau+b}{c \tau+d} \right) = \epsilon (c \tau+d)^2 z(\tau)
\end{equation*}
once and the relation
\begin{equation*}
\Psi_n \left(x(\tau), x\left( \frac{\alpha \tau+\beta}{\delta} \right) \right)=0
\end{equation*}
twice and then set $\tau$ to the value in the hypothesis of the theorem.
\end{proof}


\begin{corollary} \label{ram_sato}
If $A_k$ is the sequence defined in Corollary \ref{cor_Zx_coeff},
\begin{align*}
\frac{\sqrt{11}}{\pi}                         &= \sum_{k=0}^{\infty} A_k \frac{307+748k}{ (-21)^{k+2}}\text{,}\\
\frac{2 \sqrt{154 \sqrt{17}-634}}{\pi}        &= \sum_{k=0}^{\infty} A_k \frac{1779-195 \sqrt{17}+3040 k}{ (-22-7 \sqrt{17})^{k+2}}\text{,}\\
\frac{214 \sqrt{119}-882 \sqrt{7}}{\pi}       &= \sum_{k=0}^{\infty} A_k \frac{9241-1047 \sqrt{17}+21280k}{ (-90-21 \sqrt{17})^{k+2}}\text{,}\\
\frac{\sqrt{1041894 \sqrt{17}-4295839}}{\pi}  &= \sum_{k=0}^{\infty} A_k \frac{71065-15096 \sqrt{17}+50740k}{ (-345-84 \sqrt{17})^{k+2}}\text{,}\\
\frac{9 \sqrt{2038550094 \sqrt{17}-8405157343}}{\pi}              &= \sum_{k=0}^{\infty} A_k \frac{74004567-11655082 \sqrt{17}+178775028k}{ (-1025-252 \sqrt{17})^{k+2}}\text{,} \\ 
\frac{\sqrt{14(1267990301 \mp 85084065 i \sqrt{7})}}{\pi}  &= \sum_{k=0}^{\infty} A_k \frac{3370317797 \pm 95119383i\sqrt{7}+12974719520k}{161874(30\pm 33i \sqrt{7})^{k}}\text{,} 
\end{align*}
and
\begin{align*}
\frac{\sqrt{9 \sqrt{17}-37}}{\pi }    &= \sum_{k=0}^{\infty} A_k \frac{32-3 \sqrt{17}+32 k}{(12+3\sqrt{17})^{k+2}}\text{,}\\
\frac{261 \sqrt{5}-135 \sqrt{17}}{\pi} &= \sum_{k=0}^{\infty} A_k \frac{21500-788 \sqrt{85}+54720 k}{(29+4 \sqrt{85})^{k+2}}\text{,}\\
\frac{539 \sqrt{6}\mp 735 \sqrt{3}}{\pi} 2^{(1\mp 1)/2} &= \sum_{k=0}^{\infty} A_k \frac{58962 \mp7226 \sqrt{2}+199920 k}{(55\pm24 \sqrt{2})^{k+2}}.
\end{align*}

\end{corollary}
\begin{proof}
The first five series may be derived by setting $\tau=\frac{1}{2}+\frac{1}{2} \sqrt{\frac{n}{17}} i$ and using $\frac{17\tau-9}{34\tau-17}=\frac{\tau+(n-1)/2}{n}$ with $\epsilon=-1/17$ and $\eta=-1$ in Theorem \ref{thm_pi_main} for the values $n=11,19,35,59,83$. The subsequent pair may be derived from Theorem \ref{thm_pi_main} by setting $\tau = (\pm 7 +\sqrt{427}i)/34$, using $(11\tau-2)/(17\tau-3)=(\tau+69)/107$ and $(13\tau+3)/(17\tau +4) = (\tau +82)/107$, respectively. The next three arise from setting $\tau=\sqrt{\frac{n}{17}} i$ and using $\frac{-1}{17\tau}=\frac{\tau}{n}$ for $n=2,5,6$. The final series comes from setting  $\tau=\sqrt{3/34}i$ and using $\frac{-1}{17\tau}=\frac{2\tau}{3}$.
\end{proof}


\begin{theorem} \label{last}
  There are precisely eleven $\Gamma_{0}(17)$ inequivalent algebraic $\tau$ in the upper half plane such that $[\Bbb Q(x(\tau )): \Bbb Q]
  \le 2$ with $x(\tau)$ in the radius of convergence of Corollary
  \ref{cor_Zx_coeff}. 
\end{theorem}

\begin{proof}
  We formulate a complete list of algebraic $\tau$ such that $[\Bbb Q(x(\tau )):
  \Bbb Q] \le 2$ using well known facts about the $j$ invariant
  \cite{MR1513075}. First, for algebraic $\tau$, the only algebraic values of $j(\tau)$
  occur at $\Im \tau>0$ satisfying $a\tau^{2}
  + b\tau + c = 0$ for $a,b,c \in \Bbb Z$, with $d =
  b^{2} - 4ac <0$. Moreover, $[\Bbb Q(j(\tau)):\Bbb Q] = h(d)$, where
  $h(d)$ is the class number. Since there is a polynomial relation
  $P(x,j)$ between $x$ and $j$ of degree $2$ \cite[Remark
  1.5.3]{MR1400423}, we have 
  $$
  [\Bbb Q(j(\tau)):\Bbb Q]\leq[\Bbb Q(j(\tau),h(\tau)):\Bbb Q] \le  2 [\Bbb
  Q(x(\tau)):\Bbb Q],$$ and so values $\tau$ with  $[\Bbb Q(x(\tau )): \Bbb Q]
  \le 2$ satisfy $[\Bbb Q(j(\tau)):\Bbb Q] =h(d) \le 4$. Therefore, the
  bound $|d| \le 1555$ for $h(d) \le 4$ from \cite{MR2031415}
  implies that the following algorithm results in a complete list of algebraic
  $(\tau, x(\tau))$ with $[\Bbb Q(x(\tau )): \Bbb Q] \le 2$: \\
  \indent For each discriminant $-1555 \le d \le -1$, 
  \begin{enumerate}
  \item List all primitive reduced $\tau = \tau(a,b,c)$ of
    discriminant $d$ in a fundamental domain for $PSL_{2}(\Bbb Z)$. Translate these values via a set of coset representatives for $\Gamma_{0}(17)$ to a fundamental domain for $\Gamma_{0}(17)$. 
\item Factor the resultant of $P(X,Y)$ and the class polynomial
  $$H_{d}(Y) = \prod_{\substack{(a,b,c)\text{ reduced, primitive} \\
      d = b^{2}-4ac}} \left ( Y - j\Bigl (
  \frac{-b+\sqrt{d}}{2 a} \Bigr ) \right ).$$  The linear and quadratic factors
  of the resultant correspond to a complete list of $x = x(\tau)$, for
  $\tau$ of discriminant $d$, such that $[\Bbb Q(x(\tau )): \Bbb Q]
  \le 2$. Associate candidate values $\tau$ from Step 1 to $x$ by
  numerically approximating $x(\tau)$. For each tentative pair,
  $(\tau, x)$, prove the evaluation $x=x(\tau)$ as indicated in the proof of Theorem \ref{expl}. 
  \end{enumerate}
Values $\tau$ with
$x(\tau)$ in the radius of convergence of Corollary
\ref{cor_Zx_coeff} appear in Table \ref{tab_x_values_17}.
\end{proof}

%

\bibliographystyle{plain}
\bibliography{references}

\end{document}